\documentclass[12pt]{article}
\usepackage{amsmath,amsthm,amssymb, amstext, amsopn, amsxtra}

\usepackage{latexsym}
\usepackage{amsfonts}
\usepackage{amsmath}
\usepackage{amsthm}
\usepackage{amsfonts}
\usepackage{amssymb}
\usepackage{amscd}
\usepackage{cite}
\usepackage{xcolor}

\newcommand\blfootnote[1]{%
	\begingroup
	\renewcommand\thefootnote{}\footnote{#1}%
	\addtocounter{footnote}{-1}%
	\endgroup
}

\usepackage{graphicx}

\theoremstyle{plain}
\newtheorem{thm}{Theorem}[section]
\newtheorem{lemma}[thm]{Lemma}
\newtheorem{cor}[thm]{Corollary}
\newtheorem{prop}[thm]{Proposition}
\newtheorem*{alt-thm}{Theorem}
\newtheorem*{alt-cor}{Theorem}

\theoremstyle{definition}
\newtheorem{definition}[thm]{Definition}

\newtheorem{question}[thm]{Question}
\newtheorem*{remark}{Remark}
\makeatletter
\renewenvironment{proof}[1][\proofname]{%
	\par\pushQED{\qed}%
	\normalfont 
	\topsep6\p@\@plus6\p@\relax
	\trivlist
	\item[\hskip\labelsep\bfseries #1\@addpunct{.}]%
}{%
	\popQED\endtrivlist\@endpefalse
}
\makeatother

\renewcommand{\proofname}{Proof}

\begin{document}

	\begin{center}
		\Large{\bf Rings in which one-sided strongly $\pi$-regular elements are strongly $\pi$-regular }
		
		\bigskip
		\large{Dimple Rani Goyal\footnote{The work of the first author is supported by a UGC grant and will form a part of her Ph.D. dissertation
				under the supervision of the second author.} and Dinesh Khurana\footnote{The work of the second author is supported by DST FIST Grant grant SR/FIST/MS-II/2019/43 dated 7-1-2020.}}
		
	\end{center}
	
	\vspace{2mm}

	\begin{abstract}
		{\footnotesize In 1977, Hartwig and Luh asked: If $a$ is an element in a Dedekind-finite ring $S$, then does $aS = a^2S$ imply $Sa = Sa^2$? This question was answered negatively by Dittmer, Khurana, and Nielsen in 2014. On the other hand, Dittmer et al.\ proved that the question of Hartwig and Luh has a positive answer for Dedekind-finite exchange rings. We explore the question of Hartwig and Luh for various other classes of Dedekind-finite rings. We will also prove that the condition in question is not left-right symmetric.}
	\end{abstract}
	\blfootnote{MSC Primary: 16E50, Secondary 16U99}\\
	\blfootnote{Keywords: Left strongly $\pi$-regular elements, Right strongly $\pi$-regular elements, Strongly $\pi$-regular elements, Dedekind-finite rings, Right Dischinger rings, Left Goldie rings, UU rings, Suitable elements, Exchange elements}

	\vspace{-5mm}
	
	\section{Introduction}
	 Kaplansky asked the following question in his 1950 paper \cite{KP}:\\
	 
	 \noindent {\em It is natural to ask what can be deduced from just the assumption $a^{n+1} x = a^n$. For example, does it enable one to construct idempotents?}\\
	 
	 \par Following the literature, an element $a$ in a ring $R$ is \emph{right strongly $\pi$-regular}, if there exists some $x \in R$ such that $a^n= a^{n+1}x$ for some integer $n \geq 1$.\ Kaplansky was interested in this condition because in rings with identity, this definition is equivalent to termination of the descending chain of principal right ideals $aR\supseteq a^2R\supseteq a^3R\supseteq \ldots$. Similarly, \emph{left strongly $\pi$-regular} elements are defined. Moreover, if an element $a$ in a ring $R$ is both left and right strongly $\pi$-regular, then $a$ is \emph{strongly $\pi$-regular}.\\
	 
	 On the other hand if $a = a^2x$ for $a,\,x \in R$, then $a$ is said to be \emph{right strongly regular}. The \emph{left strongly regular} elements are defined similarly. An element $a \in R$ which is both left and right strongly regular is called \emph{strongly regular}. It is clear that if $a$ is strongly regular, then $R = aR \oplus r_R(a)$ implying that every strongly regular element is (von Neumann) regular. It is easy to see that an element is strongly $\pi$-regular iff its some power is strongly regular. As it is common in ring theory to add $\pi$ when talking about conditions that hold for some power, it would have been more appropriate to call strongly $\pi$-regular element as $\pi$-strongly regular. \\
	 
	 It is quite easy to see that a strongly $\pi$-regular element leads to an idempotent, as desired by Kaplansky in the above question, as its some power is regular. But the  Kaplansky's question was for rings whose every element is right strongly $\pi$-regular.
	When Kaplansky raised this question, it was not known that if all elements of a ring are right strongly $\pi$-regular, then all its elements are also left strongly $\pi$-regular. This amazing result was proved about $25$ years later by Dischinger \cite{Dischinger} in 1976. This result of Dischinger \cite{Dischinger} leads to a very satisfactory answer to the question raised by Kaplansky.\\
	
	 However, this result of Dischinger does not hold locally{\footnote{The term “locally’’ is used in the sense of considering the result element-wise}}, and a right strongly $\pi$-regular element may not be left strongly $\pi$-regular. For instance, if $R$ is not Dedekind finite and has two elements $a, \,b$ such that $ab = 1$ but $ba \neq 1$, then $a$ is right strongly $\pi$-regular but not left strongly $\pi$-regular. On the other hand, Azumaya \cite[Theorem 4]{Az} proved that if $R$ has a bounded index of nilpotence, then every one-sided strongly $\pi$-regular element is strongly $\pi$-regular. Azumaya \cite[Lemma 3]{Az} also proved that if $a \in R$ is strongly $\pi$-regular with $a^n = a^{n+1}x$ and $a^m = ya^{m+1}$ for some $x,\,y \in R$ and $n,\,m \in \mathbb{N}$, then $a^m = a^{m+1}x$ and $a^n = ya^{n+1}$ also. In view of this, the following result follows.
	
	\begin{prop}\label{D.Def}
		For a ring $R$, the following conditions are equivalent.\\[2mm]
		\textup{(1)} Every right strongly $\pi$-regular element of $R$ is strongly $\pi$-regular element.  \\[2mm]
		\textup{(2)} Every right strongly regular element of $R$ is strongly regular.
	\end{prop}
\begin{proof} Suppose $(1)$ holds and $a \in R$ such that $a=a^2x$, then by assumption, $a$ is strongly $\pi$-regular element i.e., there exists an element $y \in R$ such that $a^n=ya^{n+1}$ for some integer $n \geq 1$. By \cite[Lemma 3]{Az}, we have $a=ya^2$.\\
 Next, suppose $(2)$ holds and assume that  $a \in R$ such that $a^n=a^{n+1}x$ for some integer $n \geq 1$ and $x \in R$. Since $a^n=a^{n+1}x=a^{2n}x^n$, thus by assumption, $a^n$ is strongly regular. This implies that there exists an $y \in R$ such that $a^n=ya^{2n}$ for some integer $n \geq 1$. Therefore, $a$ is strongly $\pi$-regular. 
\end{proof}
	\begin{definition}
		We call a ring $R$ to be {\em right Dischinger} if it satisfies either of the equivalent conditions of Proposition \ref{D.Def}. {\em Left Dischinger} rings are defined similarly. A ring which is both left and right Dischinger will be simply called a {\em Dischinger ring}. In other words, Dischinger rings are those in which Dischinger's result also holds locally. 
	\end{definition} 

It is well known that Drazin-invertible elements are precisely the strongly $\pi$-regular elements. Thus in a right Dischinger ring every right strongly $\pi$-regular element is Drazin invertible. Note that the definition of one-sided Dischinger rings is still valid for rings without identity. Throughout this paper, all rings are assumed to be associative but may not have an identity. However, if $R$ possesses an identity, then it will be denoted by $1$.
	It is clear that if $R$ is a one-sided Dischinger ring with identity, then $R$ must be Dedekind-finite.\ Hartwig and Luh \cite[Page 94]{HartLuh} in 1977 asked whether a Dedekind-finite ring is right Dischinger. A strong negative answer to this question was given by Dittmer, Khurana, and Pace in \cite[Theorem 2.3]{Dit} by constructing an example of a ring with only trivial idempotents that is not right Dischinger.\ To be specific, it was proved in  \cite[Theorem 2.3]{Dit}  that the ring 
	\begin{equation}
		R = F\langle a,x :  a= a^2x\rangle, \tag{1.3}
	\end{equation}  
	where $F$ is a commutative ring with trivial idempotents, has only trivial idempotents. On the positive side, it was proved in \cite[Corollary 3.5]{Dit} that a Dedekind-finite exchange ring is a Dischinger ring.\\
	\par Now, since it is known that a Dedekind-finite ring may not be right Dischinger, but a Dedekind-finite exchange ring is Dischinger, it is natural to ask which other classes of Dedekind rings are (one-sided) Dischinger.\\
\par It is well-known that if in a ring $R$ there exist $a$ and $b$ such that $ab=1$ but $ba \neq 1$, then $ E_{ij}=b^i(1-ba)a^j$, where $i,j \in \mathbb{N}$ is an infinite set of matrix units. In view of this, the rings in which nilpotent elements satisfy certain properties, such as, nilpotent elements form an ideal (such rings are called NI) or a subring (such rings are called NR), are always Dedekind-finite. Azumaya in \cite[Theorem 4]{Az} already proved that if $R$ has a bounded index of nilpotence, then every right strongly regular element is strongly regular. It is quite interesting to see under which conditions on the set of nilpotent elements one-sided strongly $\pi$-regular elements becomes strongly $\pi$-regular. Also, almost a decade ago, through private communication, Nielsen asked one of the authors whether Dischinger's result holds locally in NI rings or more generally, in rings in which the set of nilpotent elements is closed under addition or ring is right linearly McCoy. Also, Azumaya (see \cite[Theorem 2]{Az}) proved that a right Noetherian ring is Dischinger. In light of these considerations, we may pose this question in well-known classes of Dedekind-finite rings: weakly semicommutative, one-sided duo, one-sided Goldie, one-sided Noetherian, NI, etc. We also generalize the result of Azumaya \cite[Theorem 4]{Az} that rings with bounded index of nilpotence are Dischinger in Theorem \ref{D.GenAzumaya}.  Also, by characterizing units, nilpotents, zero-divisors, etc.\ in the ring appearing in (1.3), we can show that a large number of properties do not yield the Dischinger condition, such as, NR, UU (rings in which units are unipotent), one-sided linearly McCoy, etc. Also, we prove that rings of the form (1.3) helps us to answer that Dischinger condition is not left-right symmetric i.e., one-sided Dischinger ring may not be Dischinger. In the end, we also raise several interesting questions for further investigation.

\section{ Some classes of right Dischinger rings}
In this section, besides studying some properties of right Dischinger rings, we also give some classes of rings that are right Dischinger. The first two results are very basic and follow immediately.\ We will use the following result, which is immediate from Azumaya's \cite[Lemma 3]{Az} mentioned in Section 1, frequently. 
	\begin{lemma}\label{D.Lemma1}
Let $a$ be an element in a ring $R$ such that $a=a^2x$ and $a$ is left strongly $\pi$-regular, then $a$ is strongly regular.
	\end{lemma}
	
	Dittmer et al.\ \cite[Lemma~3.3]{Dit} proved that if $a$ is a regular element in a Dedekind-finite ring $R$ and $aR = a^{2}R$, then $Ra = Ra^{2}$. Using Azumaya's result \cite[Lemma~3]{Az}, we can extend this result easily as follows.

	\begin{lemma}\label{D.Nasc2}
		Let $a$ be an element in a Dedekind-finite ring $R$ such that $a = a^2x$ for some $x \in R$. If $a$ is $\pi$-regular, then $Ra=Ra^2$.
	\end{lemma}
	\begin{proof}
		Suppose $a^n$ is regular, that is, for some $y \in R$, $a^n=a^nya^n$. As $a=a^2x$ then $a^n=a^{n+1}x=a^{2n}x^n$. By \cite[Lemma 3.3]{Dit}, $a^n$ is strongly regular, that is, $a$ is strongly $\pi$-regular. By Lemma \ref{D.Lemma1}, $a$ is strongly regular.
	\end{proof}
	The proof of \cite[Theorem 3.4]{Dit} yields the following useful result.
	\begin{lemma}\label{D.DKP}
		Let $a$ be an element in a Dedekind-finite ring $R$. If there exists an idempotent $e\in R$ such that $aR + eR = R$ and $r_R(a) + (1-e)R = R$, then $a$ is unit-regular.
	\end{lemma}
	
	Recall that an element $a\in R$ is called {\em suitable} if there exists an idempotent $e \in aR$ such that $1-e \in (1-a)R$. It was proved in \cite{TAMS} that being suitable is a left-right symmetric property. An element $a\in R$ is called right exchange if for any right ideal $I$ of $R$  with $aR + I = R$, there exists an idempotent $e \in aR$ such that $1-e \in I$. Nicholson \cite[Proposition 1.6]{Nic77} proved that a regular element is suitable. Using the same idea, it is easy to prove that $\pi$-regular elements are also suitable.
	\begin{prop}\label{D.NaSc}
		Let $a$ be an element in a Dedekind-finite ring $R$ such that $a = a^2x$ for some $x \in R$. Then the following are equivalent.\\[2mm]
		\textup{(1)} $Ra=Ra^2$.\\[2mm]
		\textup{(2)} $ax$ is a suitable element. \\[2mm]
		\textup{(3)} $ax$ is an exchange element.\\[2mm]
		\textup{(4)} $ax$ is a $\pi$-regular.\\[2mm]
		\textup{(5)} $ax$ is a regular element. \\[2mm]
		\textup{(6)} $ax$ is an idempotent.\\[2mm]
		\textup{(7)} $axa = a$. 
	\end{prop}
	\begin{proof}
		$(7) \Rightarrow (1)$ can be obtained from \cite[Lemma 3.3]{Dit}. Now, we prove $(1) \Rightarrow (7)$. Suppose $a = ya^2$. Then $axa = ya^2xa = ya^2 = a$. 
		
		The implications $(7) \Rightarrow (6) \Rightarrow (5) \Rightarrow (4)$ and $ (3) \Rightarrow (2)$ are obvious.  Also, $(4) \Rightarrow (2)$ is easy to prove. Lastly, we prove $(2) \Rightarrow (1)$. Suppose $ax$ is a suitable element. So there exists an idempotent $1-e \in axR$ such that $e \in (1-ax)R$. Now $a = a^2x$ implies that $a(1-ax) = 0$ and so $(1-ax)R \subseteq r_R(a)$. Now $1-e \in axR \subseteq aR$ and $e \in (1-ax)R \subseteq r_R(a)$ implies that $aR + eR = R$ and $r_R(a) + (1-e)R = R$. So by Lemma \ref{D.DKP}, $a$ is unit-regular and so by \cite[Lemma 3.3]{Dit}, we have $Ra = Ra^2$. 
	\end{proof}
	
	The set of all nilpotent elements in a ring $R$ is denoted by $\rm{N}$$(R)$. A ring $R$ is called {\em weakly semicommutative} if for any $a,\,b\in R$ with $ab = 0$ implies $aRb \subseteq \rm{N}$$(R)$. 
	\begin{cor}
		A weakly semicommutative ring with identity is Dischinger. 
	\end{cor}
	\begin{proof}
		Let $R$ be a weakly semicommutative ring. It is enough to show that $R$ is right Dischinger. Suppose $a = a^2x$ for some $a,\,x \in R$. As $a(1-ax)=0$, so $ax(1-ax) \in \rm{N}$$(R)$. This implies that $ax$ is strongly $\pi$-regular. So in view of Proposition \ref{D.NaSc}, it is enough to show that $R$ is Dedekind-finite. Let $xy = 1$ for some $x,\,y \in R$ then $x(1-yx) = 0$. Since $R$ is weakly semicommutative, $xy(1-yx) = 1-yx$ is nilpotent. Also, $1-yx$ is an idempotent. This can happen only if $1-yx = 0$, that is, $yx=1$. 
	\end{proof}
	Let $R$ be a ring with identity in which $a \in \rm{N}$$(R)$ and $r \in R$ imply that $ar \in \rm{N}$$(R)$. Then $R$ is weakly semicommutative. Indeed suppose $ab = 0$ for some $a, \, b \in R$. Then $ba \in \rm{N}$$(R)$ and so $bar \in \rm{N}$$(R)$ for every $r \in R$. This implies that $arb \in \rm{N}$$(R)$ for every $r \in R$. We have repeatedly used the fact that $xy \in \rm{N}$$(R)$ iff $yx \in \rm{N}$$(R)$. 
	\begin{cor}\label{D.NI}
		Let $R$ be ring with identity in which for any elements $a \in \rm{N}$$(R)$ and $r \in R$ implies that $ar \in \rm{N}$$(R)$ (equivalently, $ra \in \rm{N}$$(R)$). Then $R$ is Dischinger. In particular, an $NI$ ring with identity is Dischinger.
	\end{cor} 
	
	In light of Corollary \ref{D.NI}, one may question if an NR ring is Dischinger. We will show in Section 3 (Corollary \ref{D.nr1}) that this is not true.
	
	\vspace{3mm}
	If $R$ is left duo (i.e., $Ra$ is a right ideal for any $a \in R$) and $a \in \rm{N}$$(R)$, it is easy to see that $ra \in \rm{N}$$(R)$ for any $r\in R$. Indeed if $a^n=0$, then as $aR \subseteq Ra$, so  $(ar)^n \in Ra^nR = 0$. So by Corollary \ref{D.NI}, a one-sided duo ring is Dischinger. A ring $R$ is said to be \emph{weakly left duo} if for every $a \in R$, there exists an $n \in \mathbb{N}$ such that the left ideal $Ra^n$ is a two-sided ideal. Next, we establish another sufficient condition (may not be necessary) for a Dedekind-finite ring to be right Dischinger, which will be helpful to prove that weakly left duo rings are right Dischinger.
	\begin{thm}\label{D.wld}
		Let $R$ be a Dedekind-finite ring and $a \in R$ such that $a=a^2x$ for some $x \in R$. If $Ra^n$ is an ideal for some $n\in \mathbb{N}$, then $a$ is strongly regular. 
	\end{thm}
	\begin{proof}
		As $a^n \in Ra^n$ and $Ra^n$ is an ideal, this implies $a^nx^{n-1} = a \in Ra^n$. If $n > 1$, there is nothing to prove. Suppose $n=1$, this implies $Ra$ is an ideal, so $ax=ya$ for some $y \in R$. On left multiplication this with $a$ gives $a=aya$. As $R$ is Dedekind-finite, by \cite[Lemma 3.3]{Dit}, $a$ is strongly regular. 
	\end{proof}
	\begin{cor}\label{D.duo}
		Weakly left duo rings with identity are right Dischinger.     
	\end{cor}
	\begin{proof}
		In view of the above result, i t is enough to show that $R$ is Dedekind-finite. Let $xy = 1$ for some $x,\, y \in R$. As $yx$ is an idempotent, so for every $n \in \mathbb{N}$ we have $R(yx)^n = Ryx$ is an ideal of $R$. In particular, $yxy = y \in Ryx$, implying that $R = Ry \subseteq Ryx \subseteq Rx$. Thus $Rx = R$ and so $R$ is Dedekind finite.
	\end{proof}
	Further, if $R$ is an abelian ring, then the converse of Theorem \ref{D.wld} is also true. 
	\begin{thm}
		Let $R$ be an abelian ring with identity and $a \in R$ such that $a=a^2x$ for some $x \in R$. Then $Ra^n$ is an ideal for some $n\in \mathbb{N}$ if and only if $a$ is strongly regular. In addition, $Ra^n$ is an ideal for every $n \in \mathbb{N}$.
	\end{thm}
	\begin{proof}
		Since $a$ is strongly regular, we have $Ra=Re$ for some idempotent $e \in R$. As $R$ is an abelian ring, $Ra$ is an ideal. Hence, for any $r \in R$, $a^nr \in Ra^n$ for any $n \in \mathbb{N}$.
	\end{proof}
	However, it is not necessary that for a strongly regular element $a \in R$, $Ra^n$ is an ideal for some $n \in \mathbb{N}$. For example, if $R=\mathbb{M}_2(\mathbb{Q})$. Then $a=E_{11}$ is a strongly regular element in $R$, but $Ra^n=Ra$ is not an ideal for any $n\in \mathbb{N}.$

	Azumaya \cite[Theorem 4]{Az} proved that a ring with bounded index of nilpotence is Dischinger. In the following, we generalize this result.
	\begin{thm}\label{D.GenAzumaya}
		A ring $R$ with identity satisfying any of the following two conditions is right Dischinger.\\[2mm]
		\textup{(1)} There is no infinite ascending chain of principal right ideals $a_1R \subseteq a_2R \subseteq \ldots$ such that each $a_i$ is nilpotent and the index of nilpotence of each $a_i$ is less than that	of $a_{i+1}$.\\[2mm]
		\textup{(2)} There is no infinite descending chain $l_R(a_1) \supseteq l_R(a_2) \supseteq \ldots$ such that each $a_i$ is nilpotent and the index of nilpotence of each $a_i$ is less than that	of $a_{i+1}$.
	\end{thm}
	\begin{proof}
		(1) We first show that $R$ is Dedekind-finite. Suppose, to the contrary, there exist $b,c \in R$ such that $bc = 1$ but $cb \neq 1$. Then for any $n \in \mathbb{N}$, we show that the element $d_n := b-c^nb^{n+1}$ is nilpotent of index $n+1$. Since $b^{n+1}d_n=0$ thus $d_n^2=bd_n$. So we can write $$(d_n)^k = b^{k-1}d_n  \quad \text{for all} \quad k \in \mathbb{N}.$$ Therefore, $$(d_{n})^n = b^{n-1}d_n = b^{n-1}(b-c^nb^{n+1}) = b^n - cb^{n+1} \neq 0,$$ because $b^n - cb^{n+1} = 0$ implies $(b^n - cb^{n+1})c^n = 0$ and so $cb = 1$, a contradiction. Also $(d_{n})^{n+1} = b^nd_n = 0$. Now we show that $d_nR \subseteq d_{n+1}R$. As $b^n(1-c^nb^n) = 0$, so we have  $$d_{n+1}(1-c^{n+1}b^{n+1}) = (b-c^{n+1}b^{n+2})(1-c^{n+1}b^{n+1}) = b(1-c^{n+1}b^{n+1}) = b - c^{n}b^{n+1} = d_n.$$ So $d_nR \subseteq d_{n+1}R$. If $d_nR = d_{n+1}R$, then as $b^nd_n = 0$, so $b^nd_{n+1} = 0$. But this implies $0 = b^n(b-c^{n+1}b^{n+2}) = b^{n+1} - cb^{n+2}$ and on right multiplication by $c^{n+1}$ , we get $1 = cb$, a contradiction. This implies that  $d_nR \subset d_{n+1}R$, which is a contradiction to the assumption. Thus $R$ is Dedekind-finite. 
		
		Now, suppose that $a = a^2x$ for some $a,\,x\in R$. Then for $k_n := a - ax^na^n$, where $n \in \mathbb{N}$, we have $k_{n}^{n+1} =0$. Note that for any $1 \leq m \leq n$, we have  $$k_n^m = a^{m-1}k_n= a^m-ax^{n-m+1}a^n.$$ If for some $m < n$, $k_{n}^{m} = 0$, we have $a^m \in Ra^n$, i.e. $a$ is strongly $\pi$-regular so by Lemma \ref{D.Lemma1}, $a$ is strongly regular and we are through. Further, if $k_n^n=0$, then $a^n$ is regular. By Lemma \ref{D.Nasc2}, $a$ is strongly regular. Now, assume that each $k_n$ is nilpotent of index $n+1$. Note that $$k_n = k_{n+1}(1-x^na^n).$$ Hence, the chain of right ideals $k_1R \subseteq k_2R \subseteq \ldots$ must terminates by assumption. Hence, $k_nR = k_{n+1}R$. This further implies, $$0= a^nk_n=a^nk_{n+1}=a^n(a - ax^{n+1}a^{n+1}) = a^{n+1} - a^{n+1}x^{n+1}a^{n+1}.$$  So $a^{n+1}$ is regular and thus by Lemma \ref{D.Nasc2}, we have $Ra = Ra^2$. \\[3mm]
		(2) This proof runs on similar lines to that of part (1) above. So we only give a sketch of the proof. First, we show that $R$ is Dedekind-finite. If $bc = 1$ but $cb \neq 1$ for some $b,\,c \in R$, then $l_R(d_1) \supseteq l_r(d_2) \supseteq \ldots$, where $d_n:= b-c^nb^{n+1}$ is as in (1) above, is a proper descending chain as $b^n \in l_R(d_n) \verb=\= l_R(d_{n+1})$. As each $d_n$ is nilpotent of index $n+1$, this is a contradiction. So $R$ is Dedekind-finite. 
		
		Now suppose $a = a^2x$ for some $a,\,x\in R$.  Then $l_R(k_1) \supseteq l_R(k_2) \supseteq \ldots$, where $k_n := a - ax^na^n$ is as in part (1). If for some $m \leq n$, $k_{n}^{m} = 0$, then $a$ is strongly regular as in (1). If for some $n \in \mathbb{N}$, we have $a^n \in l_R(k_{n+1})$, then $a$ is $\pi$-regular and thus by Lemma \ref{D.Nasc2}, $Ra = Ra^2$. Otherwise $l_R(k_1) \supseteq l_R(k_2) \supseteq \ldots$ is a proper chain because $a^n \in l_R(k_n) \verb=\= l_R(k_{n+1})$, which is not possible.
	\end{proof}
	
	The following result can be easily obtained from Azumaya \cite[Theorem 2]{Az}.
	\begin{thm}\label{D.leftacc}
		Any ring with identity satisfying the ascending chain condition on left annihilators of subsets of $R$ is left Dischinger. In particular, left Noetherian rings are left Dischinger.
	\end{thm}
	It is natural to ask whether a ring satisfying the ascending chain condition on left annihilators of subsets is right Dischinger. In Section 3, we will show that this may not be. Moreover, we prove that with some additional condition, a ring satisfying the condition in Theorem \ref{D.leftacc} becomes right Dischinger (see Corollary \ref{D.goldie}). For this, we need the following lemma. We will tacitly use the fact that if $ab$ is nilpotent, then so is $ba$ for any $a,\,b \in R$.
	\begin{lemma}\label{D.Nra}
		Let $a$ be an element in a ring $R$ and $\rm{N}$$^r_R$$(a)$ be the subset of nilpotent elements such that for every $r \in \rm{N}$$^r_R$$(a)$ there exist elements $r_i \in R$, $i \in \mathbb{N}$ such that $r=r_1a$, $ar_i =  r_{i+1}a$ and $ar_k = 0$ for some $k \in \mathbb{N}$. Then $\rm{N}$$^r_R$$(a)$ is a nil subring of $R$.
	\end{lemma}
	\begin{proof}
		It is easy to see that $\rm{N}$$^r_R$$(a)$ is closed under addition. Let $r,\, s \in \rm{N}$$^r_R$$(a)$.  Then there exist $r_i,\,s_i \in R$ for $i \in \mathbb{N}$ such that $r=r_1a$, $ ar_i =  r_{i+1}a$, $s=s_1a$, and  $as_i =  s_{i+1}a$ for all $i \in \mathbb N$. Then $ar_ias_i = r_{i+1}as_{i+1}a$ for all $i$. Let $k$ be the least positive integer such that one of the $ar_k$ or $as_k$ is zero. Then $ar_kas_k = 0$ implying that $r_1as_1a \in \rm{N}$$^r_R$$(a)$.
	\end{proof}
	Analogously, we can define the left version of $\rm{N}$$^r_R$$(a)$for any element $a \in R$ and denote it by $\rm{N}$$^{\ell}_R$$(a)$
	\begin{thm}\label{D.Nra2}
		A ring $R$ in which $\rm{N}$$^r_R$$(a)$ is a nilpotent subring for every right strongly regular element $a$, then $R$ is right Dischinger.
	\end{thm}
	\begin{proof}
		It is enough to prove that $R$ is right Dischinger. Suppose $a = a^2x$ for some $a,\,x \in R$. Since $\rm{N}$$^r_R$$(a)$ is nilpotent subring so $(\rm{N}$$^r_R$$(a))^k = 0$, where $k \in \mathbb{N}$.  As $a^2~-~a^{n+1}x^na = 0$, so by Lemma \ref{D.Nra}, $a^2-x^na^{n+2} \in\rm{N}$$^r_R$$(a)$ for any $n \in \mathbb{N}$. Therefore, $(a^2-x^na^{n+2})^k= 0$. As $a^{n+2}~(a^2~-~x^na^{n+2}) = 0$, this implies $ (a^2-x^na^{n+2})^2=a^2(a^2-x^na^{n+2})$. Therefore, we have $$0=(a^2-x^na^{n+2})^k=(a^2)^{k-1}(a^2-x^na^{n+2})=a^{2k}-a^{2k-2}x^na^{n+2}.$$ For $n=2k-1$, the above equation gives
		
		$$a^{2k}=a^{2k-2}x^{2k-1}a^{2k+1}$$ implying that $a$ is right strongly $\pi$-regular. So $a$ is strongly regular by Lemma \ref{D.Lemma1}.	
	\end{proof}
	
	A ring $R$ with identity is called {\em left Goldie} if it satisfies the ascending chain condition on left annihilators and $_RR$ has finite uniform dimension. In \cite{Lanski}, Lanski proved that every nil subring of a left Goldie ring with identity is nilpotent. In view of this and Theorem \ref{D.Nra2}, we have the following result. 
	\begin{cor}\label{D.goldie}
		A ring $R$ in which every nil subring is nilpotent is Dischinger. In particular, a left Goldie ring with identity is a Dischinger.
	\end{cor}
	As a left Noetherian ring is left Goldie, so one-sided Noetherian rings with identity are Dischinger.

	\section{Some classes of rings that are not right Dischinger and left-right asymmetry}
	
	In this section,  we prove that a left Dischinger ring may not be right Dischinger. We will also find several classes of Dedekind-finite rings which are not right Dischinger. For instance, we will show that an NR ring or a ring satisfying the ascending chain condition on left annihilators or a UU ring may not be right Dischinger. This will be done by an in-depth analysis of the properties of the ring constructed in \cite{Dit} given in $(1.3)$ above for $F = \mathbb{F}_2$. Most of the properties of 
	\begin{equation} R = \mathbb{F}_2\langle a,x :  a= a^2x\rangle, \tag{3.1} \end{equation} 
	proved in this section holds even when $\mathbb{F}_2$ is replaced with a commutative domain $F$ with identity, yet we study the properties of the ring in (3.1) for notational convenience and also as it serves our purpose of proving that the classes of rings listed above are not right Dischinger. From now on, $R$ will be as in $(3.1)$ above.
	
	\vspace{3mm}
	As noted in \cite{Dit}, monomials in $R$ can be simplified using the relation $a=a^2x$ and the set of all reduced monomials forms an $\mathbb{F}_2$-module basis for $R$ (denoted by $\mathcal{B}$). Also, any monomial in the reduced form looks like $$ m= x^{i_{n}}a\cdots x^{i_{2}}ax^{i_1}a^{i_0}$$
	for some unique integers $n \geq1$, $i_1, i_2,..., i_{n-1} \geq 1$ and $i_0, i_n \geq 0$. Here, $n$ is the depth of $m$, denoted by depth($m$). Recall that for any element $r \in R$, a monomial $m$ appears in the support of $r$, denoted as $m \in \operatorname{supp}(r)$ if the coefficient of $m$ in the reduced form of $r$ is nonzero. We define the length of a monomial $m$ to be the total number of letters appearing in it, counted with multiplicity and denote it by $l(m)$. For example, $l(x^3ax^2) = 6$ and $l(xa^3) = 4$. Let $x^{i_{n}}a\cdots x^{i_{2}}ax^{i_1}a^{i_0}$ and $x^{j_{n}}a\cdots x^{j_{2}}ax^{j_1}a^{j_0}$ be two monomials in reduced form. As in \cite{Dit}, we say $$x^{i_{n}}a\cdots x^{i_{2}}ax^{i_1}a^{i_0} \prec_n  x^{j_{n}}a\cdots x^{j_{2}}ax^{j_1}a^{j_0}$$ if
	either $i_0 > j_0$ or if there exists an index $n_0$ with $1 \leq n_0 \leq n$ such that $i_k=j_k$ for all $k < n_0$, but $i_{n_0} < j_{n_0}$. As for any reduced monomial $m = x^{i_{n}}a\cdots x^{i_{2}}ax^{i_1}a^{i_0}$, we put $m(l) =x^{i_{l}}a\cdots x^{i_{2}}ax^{i_1}a^{i_0}$ for any $l \leq n$.  Also, if $m_1, m_2 \in \mathcal{B}$ have different depths $d_1$ and $d_2$ respectively and $d=min \{d_1,d_2\}$, then $m_1 \prec m_2$ if either $m_1(d) \prec_d m_2(d)$, or $m_1(d)=m_2(d)$ and $d=d_1 < d_2$. For example, $a^2 \prec a$, $x \prec x^2$, $a^2 \prec xa^2$, $a \prec x$, $xa^3 \prec x^2a^3$ etc. It was proved in \cite[Lemma 2.1]{Dit} that  `$\prec$' is a total order relation on the set of all reduced monomials. In a finite subset of monomials in the set $\mathcal{B}$, there always exists a unique maximal monomial with respect to the order $\prec$, whereas there can be more than one monomial of largest length in that set. For example, let $\mathcal{S}=\{axa^2, x^2a^2,ax\}$. Then $ax$ is the unique maximal monomial in the set $\mathcal{S}$ with respect to $\prec$, while both $axa^2$ and $x^2a^2$ have the largest length, equal to $4$.
	
	\vspace{3mm} We will repeatedly use the following result that was proved in \cite[Lemma 2.2]{Dit}.
	\begin{lemma}\label{D.mainlemma}
		If $m_1$ and $m_2$ are monomials in $\mathcal{B}$, then either $m_1m_2$ is already reduced or $m_1m_2 \prec m_2$ after reduction. 
	\end{lemma} 
	The following result will be crucial for everything that follows. The last part of this result was first proved by Nielsen and was communicated to the second author in 2012.
	\begin{thm}\label{D.mainthm}
		If $f \in R$ has a monomial in its support that is in $Rx$, then $r_R(f) = 0$. In particular,  $r_R(g) = 0$ for any $g \in R$ that has $1$ in its support.
	\end{thm}
	\begin{proof}
		Suppose $h$ is any nonzero element in $R$. We will show that $fh \neq 0$.\\
		Let $M = x^{i_{n}}a\cdots x^{i_{2}}ax^{i_1}a^{i_0}$ be the maximal monomial in the support of $h$, and $m$ be a monomial of the largest length in the support of $f$ that is in $Rx$. 
		
		If $mM$ appears only once in the support of $fh$, then we have nothing to prove. If there exist monomials $m_1$ in the support of $f$ and $m_2$ in the support of $h$ such that $mM = m_1 m_2$, then $m_1m_2$ is irreducible because otherwise, by Lemma \ref{D.mainlemma}, $mM = m_1m_2 \prec m_2 \prec M$, a contradiction. Due to the maximality of $M$, the monomial $m_2$ is a right subword of $M$. Also $m_2 \prec M$ because otherwise $mM = m_1M$ implies that $m = m_1$. So the length of $m_1$ is greater than that of $m$, and as $m$ is a monomial of the largest length in $Rx$,  $m_1 \in  Ra$. This implies $m_2 = M(k)$ for some $k \in \{1, \ldots, n\}$. Let $l \in \{1,\ldots, n\}$ be the smallest integer such that $M(l)$ is in the support of $h$. We will show that $mM(l)$ appears in the support of $fh$ only once.
		
		Suppose $mM(l) = m_3m_4$ for some monomial $m_3$ in the support of $f$ and $m_4$ in the support of $h$ such that $(m_3, m_4) \neq (m, M(l))$. First note that $m_3m_4$ is irreducible because otherwise, as $m \in Rx$, we have $M \prec mM(l) = m_3m_4  \prec  m_4$ by Lemma \ref{D.mainlemma}, a contradiction. If $m_3 = mm_0$ for some monomial $m_0$, then $m_0  \in Ra$ as $m$ is the largest monomial in the support of $f$ that is in $Rx$. But then $m_4 = M(t)$ for some $t < l$, violating the minimality of $M(l)$. Lastly, if $m_4 = m_0M(l)$ for some monomial $m_0$, then as $mM(l) \in RxM(l)$, so $m_0 \in Rx$. But this implies  $M \prec m_4$,  a contradiction.  
		So $mM(l)$ is a nonzero monomial in the support of $fh$ and so $fh \neq 0$. So $r_{ R}(f)=0$.
		
		For the last part, if $g \in R$ has $1$ in its support, then $xg$ has $x$ in its support. For any element $g_1 \in R$ if $gg_1 = 0$ then $xgg_1 = 0$ and by the first part, $g_1 = 0$. Thus $r_R(g) = 0$. 
	\end{proof} 

	\begin{cor}\label{D.rann}
		If $f\in R$ such that $r_R(f) \neq 0$, then $f \in Ra$. Moreover, for such an $f \in R$, there exists $k \in \mathbb{N}$ such that $r_R(f)= r_R(a^k)$.
	\end{cor}
	\begin{proof}
		In view of Theorem \ref{D.mainthm}, it is easy to see $f \in Ra$. Suppose $f=f_1a^k$ such that $f_1 \notin Ra$. Clearly $r_R(a^k) \subseteq r_R(f)$. Let $h$ be a nonzero element in $R$ such that $fh=0$. This implies $f_1a^kh=0$. Since $f_1 \notin Ra$, again by Theorem  \ref{D.mainthm}, $a^kh=0$ implies $h \in r_R(a^k)$. 
	\end{proof}
	Next, we identify all the nilpotent elements of the ring $R$ and also prove that they form a subring.
	\begin{thm}\label{D.nr}
		In the notations of Lemma \ref{D.Nra}, $\rm{N}$$(R)$ = $\rm{N}$$^r_R$$(a)$. In particular, $R$ is an NR ring.
	\end{thm}
	\begin{proof}
		In view of Theorem \ref{D.mainthm}, $\rm{N}$$(R) \subseteq Ra$. It is also clear that $\mathbb{F}_2[a] \cap \rm{N}$$(R) = 0$. Therefore, any nonzero nilpotent element of $R$ is of the form $r_1a$ and it must have at least one monomial in its support that is not in $\mathbb{F}_2[a]$. Then $ar_1$ is also nilpotent, and if $ar_1 = 0$, we are through. Note that any monomial of $r_1$ that was in $Rx$ will vanish from $ar_1$ either by cancellation or reduction. We continue this process of shifting $a$ to the left and in the process  $ar_k \in \mathbb{F}_2[a]$  for some $k \in \mathbb{N}$. As $ar_k$ is nilpotent, so $ar_k = 0$ implying that $\rm{N}$$(R) = \rm{N}$$^r_R$$(a)$
	\end{proof}
	\begin{cor}\label{D.nr1}
		An NR ring may not be right Dischinger.
	\end{cor}
	\begin{cor}\label{D.lann}
		Let $f$ be a nonzero element in $R$ such that $l_R(f) \neq 0$, then $a^kf = 0$ for some $k \in \mathbb{N}$.  If $n \in \mathbb{N}$ is smallest integer such that $a^nf = 0$, then $l_R(f) = Ra^n$.
	\end{cor}
	\begin{proof}
		Suppose $h$ is a nonzero element in $R$ such that $hf = 0$. As $f \neq 0$, by Corollary \ref{D.rann}, we have $h = h_1a^k$ for some $k \in \mathbb{N}$ such that $h_1 \notin Ra$. Since  $h_1 \notin Ra$ again by Corollary \ref{D.rann} , $r_R(h_1) = 0$. As $h_1a^kf = 0$, so $a^kf=0$. If $n \in \mathbb{N}$ is the smallest integer such that $a^nf=0$, then it is clear that $l_R(f) = Ra^n$. 	
	\end{proof}
	Recall that an element $s$ in a ring $S$ is called a left zero-divisor if $r_S(s) \neq 0$.
	\begin{cor}\label{D.zerodivisors}
		\textup{(1)} The set of left zero divisors of $R$ is $Ra$.\\[2mm]
		\textup{(2)} The set of the right zero divisors of $R$ is $\cup_{n = 1}^{\infty} r_R(a^n)$.
	\end{cor}
	\begin{proof}
		(1) It is clear from Corollary \ref{D.rann} and from the fact that $ra(1-ax) = 0$ for every $r \in R$.\\[3mm]
		(2) This follows from Corollary \ref{D.lann}.
	\end{proof}
	Since there is no infinite ascending chain of left ideals of $R$ of the form $Ra^n$, $n\in \mathbb{N}$, the following result is immediate from Corollary \ref{D.lann}. 
	\begin{cor}\label{D.acc}
		The ring $R$ satisfies the ascending chain condition on the left annihilators of subsets of $R$.
	\end{cor}
	In Corollary \ref{D.goldie}, we proved that a left Goldie ring is Dischinger. Hence, the following result shows that the ascending chain condition on left annihilators alone is not sufficient.
	\begin{cor}\label{D.azacc}
		A ring satisfying the ascending chain condition on left annihilators of subsets may not be right Dischinger.
	\end{cor}
	Now we have sufficient results to prove that a left Dischinger ring may not be right Dischinger. As the ring $R$ is not right Dischinger, it is enough to prove the following result.
	\begin{thm}\label{D.symm}
		$R$ is left Dischinger. Moreover, $Rf = Rf^2$, for some $f\in R$, implies either $f = 0$ or $f$ is a unit.
	\end{thm}
	\begin{proof}
		Azumaya in \cite[Theorem 2]{Az} proved that a ring $S$ which has no infinite ascending chain of the form $l_S(s) \subseteq l_S(s^2) \subseteq \ldots$, is left Dischinger. So $R$ is left Dischinger by Corollary \ref{D.acc}. Lastly, suppose $Rf = Rf^2$ for some $f \in R$. As $R$ is left Dischinger, $fR = f^2R$ also implies that $f$ is strongly regular. So $f = eu$ for some unit $u$ and idempotent $e$ in $R$. But by \cite[Theorem 2.3]{Dit} only idempotents in $R$ are $0$ and $1$, implying that either $f=0$ or $f$ is a unit. 
	\end{proof}
	We will need the following result to prove that $R$ is a UU ring. This will also lead to a self-contained proof of the last part of Theorem \ref{D.symm}.
	\begin{lemma}\label{D.aa}
		If $f + g + fg \in Ra$ for some $f,\,g \in R$ both of which don't have $1$ in their supports, then $f \in Ra$ and $g \in Ra$.
	\end{lemma}
	\begin{proof}
		It is enough to show that $g \in Ra$. Suppose it is not the case and $g = g_1x + g_2a$ for some $g_1,\, g_2 \in R$ such that $g_1 \notin Ra^2$. Now $f + g + fg \in Ra$ implies that $$f + g_1x + fg_1x \in Ra.$$ If $f \in Ra$, then we have $g_1x + fg_1x \in Ra$. If $f=\sum_{i \in I} m_i$ then there exists a subset $J \subseteq I$ such that for $f'= \sum_{j \in J}m_j$, $g_1x + f'g_1x=0$. This implies $(1+f')g_1x=0$. As $1 \in \rm{supp}(1+f')$, thus by Theorem \ref{D.mainthm}, $g_1x=0$. Since for every $k \in \mathbb{N}$, $a^kx \neq 0$, so by Corollary \ref{D.lann}, $g_1 = 0$, a contradiction. 
		
		Lastly, assume that $  f\notin Ra$. Let $M$ be the maximal monomial in support of $g$ and $m$ be a largest length monomial in support of $f$ that ends with $x$. Same as in proof of Theorem \ref{D.mainthm}, in the product $fg$ either $mM(l)$ or $mM$ survives. If $mM(l)$ survives, then $mM(l) \notin \rm{supp}(f)$ since $l($mM(l)$)>l(m)$ and $mM(l) \in Rx$. Also, $mM(l) \notin \rm{supp}(g)$, since $ M \prec mM(l)$. This implies $mM(l) \in \rm{supp}(f+g+fg) \cap Rx$, a contradiction as $f + g + fg \in Ra$ . Similarly, if $mM$ survives then $mM \in \rm{supp}(f+g+fg) \cap Rx$, again a contradiction.
	\end{proof}
	\begin{remark}
		Using Lemma \ref{D.aa}, we give a self-contained proof of the last part of Theorem \ref{D.symm}.
	\end{remark}
	
	Let $f = gf^2$ for some $f,\,g\in R$. Suppose $f \neq 0$. If $l_R(f) = 0$, then $1=gf$ and so $f$ is invertible. Assume that $l_R(f) \neq 0$. Then by Corollary \ref{D.lann}, $l_R(f) = Ra^k$, where $k \in \mathbb{N}$ is the smallest positive integer such that $a^kf = 0$. 
	
	As $gf - 1 \in  l_R(f) = Ra^k$, so $1$ is not in the support of $gf-1$. This implies that $1$ is in the support of both $f$ and $g$. Let $f = 1+ f_1$ and $g = 1 + g_1$. As $gf-1 = g_1  + f_1 + g_1f_1 \in Ra^k$, so by Lemma \ref{D.aa} both $f_1$ and $g_1$ are in $Ra$. Let $f_1 = f_2a$ and $g_1 = g_2a$. Then $f = 1+ f_2a$ and $g = 1+g_2a$. 
	
	Recall that $l_R(f) \neq 0$. If $k = 1$, then $0 = af = a(1 + f_2a)$ implies that $0 = (1+af_2)a$. As $l_R(a) = 0$, this implies that $af_2 = 1$ and so $a$ is invertible. This is a contradiction. 
	
	Now assume that $k > 1$ and that there don't exist elements $h$ and $t$ in $R$ such that $h = th^2$ with $a^{k-1}h = 0$. Then $0 = a^kf = a^k(1 + f_2a) = a^{k-1}(1 + af_2)a$. As $l_R(a) = 0$ we have $a^{k-1}(1 + af_2) = 0$. Now $f = gf^2$ implies that $$1+f_2a = (1+g_2a)(1+f_2af_2a) \Rightarrow
	a(1+f_2a) = a(1+g_2a)(1+f_2af_2a)$$  $$\Rightarrow (1+af_2)a = (1+ag_2)(1+af_2af_2)a.$$ As $l_R(a) = 0$, we have
	$$ 1+af_2 = (1+ag_2)(1+af_2af_2).$$  As $a^{k-1}(1 + af_2) = 0$, letting $h = 1 + af_2$ and $t = 1+ag_2$ we have that $h = th^2$ with $a^{k-1}h = 0$, a contradiction. \qed
	\begin{lemma}\label{D.unitprop}
		\textup{(1)} If $f,\,g \in R$ are such that $fg = 1$, then both $f$ and $g$ have $1$ in their supports.\\[2mm]
		\textup{(2)} No non-identity element of $\mathbb{F}_2[a]$ is invertible in $R$. \\[2mm]
		\textup{(3)} If $1+f$ is a unit in $R$, then $f \in Ra$.
	\end{lemma}
	\begin{proof}
		\textup{(1)} This is clear because the product of two monomials is $1$ iff both equal $1$. \\[2mm]
		\textup{(2)} Suppose to the contrary $1 + f \in \mathbb{F}_2[a]$ is invertible in $R$ with inverse $1+g$, where $f\neq 0$ and $g$ don't have $1$ in their supports.
		Also $g \in Ra$ by Lemma \ref{D.aa}. Clearly, $g \notin \mathbb{F}_2[a]$ because otherwise $1+f \in \mathbb{F}_2[a]$ is invertible in $\mathbb{F}_2[a]$. This is not possible as $1$ is the only invertible element in the ring $\mathbb{F}_2[a]$. Now $(1+g)(1+f)=1$ implies $g+f = gf$. Let $M$ be the largest monomial in the support of $g$ that is not in $\mathbb{F}_2[a]$. Then $M$ is present in the support of $g+f$. We show that $M$ is not present in the support of $gf$. Suppose $M = ma^k$, where $m$ is a monomial in the support of $g$ and $a^k$, $k \in \mathbb{N}$, is a monomial in the support of $f$. Clearly $m \neq M$. So by the maximality of $M$, we have $m \prec M$. Thus $M = ma^k \prec m \prec M$, which is a contradiction. \\[2mm]
		\textup{(3)} Suppose $1+g$ is inverse of $1+f$. Then $f + g + fg = 0$ and as $f$ and $g$ don't have $1$ in their supports, so $f\in Ra$ by Lemma \ref{D.aa}. 
	\end{proof}
	Now, we prove another main result of this section.
	\begin{thm}\label{D.UU}
		Every unit in $R$ is a unipotent. That is, $R$ is a $UU$ ring.
	\end{thm}
	\begin{proof}
		The proof follows on similar lines as that of Theorem \ref{D.nr}. Let $1+f \in \rm{U}$$(R)$.  By Lemma \ref{D.unitprop}(3), $f= f_1a$ for some $f_1\in R$. As $1+ f = 1+ f_1a$ is a unit, so is $1+af_1$. Again by \ref{D.unitprop}(3), $af_1 = f_2a$ for some $f_2 \in R$. Similarly, $1+af_2$ is a unit. Proceeding on these lines $af_k \in \mathbb{F}_2[a]$, for some $k\in \mathbb{N}$, as in the proof of Theorem \ref{D.nr}. As $1+af_k$ is a unit, so $af_k = 0$ by Lemma \ref{D.unitprop}(2). This implies that $f \in N_R(a)$ and so $f$ is nilpotent. Thus $R$ is a UU ring.
	\end{proof} 
	
	Part \textup{(1)} of the following result was noted without proof in \cite{Dit} also.
	\begin{cor}
		\textup{(1)} $\rm{J}$$(R) = 0$.\\[2mm]
		\textup{(2)} $\rm{N}$$(R) + \rm{U}$$(R) \subseteq \rm{U}$$(R).$\\[2mm]
		\textup{(3)} For any $u \in \rm{U}$$(R)$, we have $u \rm{N}$$(R) \subseteq \rm{N}$$(R)$ and $ \rm{N}$$(R)u \subseteq \rm{N}$$(R)$.
	\end{cor}
	\begin{proof}
		(1) For any nonzero $f\in R$ there exists $k\in \mathbb{N}$ such that $fx^k$ has a monomial in its support that is in $Rx$. So by Corollary \ref{D.rann}, $fx^k$ is not nilpotent and by Theorem \ref{D.UU}, $1+fx^k$ is not a unit. So $\rm{J}$$(R) = 0$.\\[2mm]
		(2) and (3) follow from the facts that $\rm{N}$$(R)$ is a subring (Theorem \ref{D.nr}) and units in $R$ are unipotents (Theorem \ref{D.UU}).
	\end{proof}
	Camillo and Nielsen in \cite{CVP} defined a ring $S$ to be right linearly McCoy if for any two linear polynomials $f,\,g\in S[x]$ with $fg = 0$, there exists a nonzero $s\in S$ such that $fs=0$. A left linearly McCoy rings can be defined similarly. As a one-sided linearly McCoy ring is Dedekind-finite \cite[Theorem 5.2]{CVP},  one may wonder if it is right Dischinger. This is not the case as we will prove that the ring $R$ is left and right linearly McCoy. We need the following result proved in \cite[Proposition 3.3]{AD}.
	\begin{lemma}
		A ring $S$ is right linearly McCoy iff for any elements $s_1,\,s_2\in S$, $$r_S(s_1) \cap r_S(s_2) = 0 \Rightarrow s_2r_S(s_1) \cap s_1r_S(s_2) = 0.$$
	\end{lemma}
	\begin{thm}
		$R$ is left and right linearly McCoy.
	\end{thm}
	\begin{proof}
		By Corollary \ref{D.rann}, for any $f \in R$ either $f \in Ra$ or $r_R(f) = 0$. Also if both $f,\,g \in Ra$, then $(1-xa)R \in r_R(f) \cap r_R(g)$ implying that $r_R(f) \cap r_R(g) \neq 0$. So $r_R(f) \cap r_R(g)=0$ precisely when either $r_R(f) = 0$ or $r_R(g) = 0$ and in this case it is clear that $gr_R(f) \cap fr_R(g) = 0$ implying that $R$ is right linearly McCoy. 
		
		Also by Corollary \ref{D.lann}, for any $f \in R$ either $l_R(f) = Ra^k$ for some positive integer $k$ or $l_R(f)=0$. So $l_R(f) \cap l_R(g) = 0$ precisely when either $l_R(f) = 0$ or $l_R(g) = 0$. So, as in the above case, $R$ is left linearly McCoy.
	\end{proof}
	\begin{cor}\label{D.lmccoy}
		A left and right linearly McCoy ring may not be right Dischinger.
	\end{cor}
	Recall that a ring $S$ is said to be right McCoy if for any two polynomials $f,\,g\in S[x]$ with $fg = 0$, there exists a nonzero $s\in S$ such that $fs=0$. A left McCoy rings can be defined similarly. It is well known that one-sided duo rings are McCoy, and we have already proved that a one-sided duo ring with identity is Dischinger. In view of this and Corollary \ref{D.lmccoy}, the following natural question arises, which we are not able to answer.
	\begin{question}
		Is a right or left McCoy ring with identity right Dischinger?
	\end{question} 
	Yao in \cite{Yao} defined the class of \emph{weakly left duo} rings as a generalization of left duo rings. In Corollary \ref{D.duo}, we proved that a weakly left duo ring with identity is right Dischinger. But we don't know the answer to the following question.
	\begin{question}
		Is a weakly left duo ring with identity left Dischinger?	
	\end{question}
	Since a weakly left (right) duo ring with identity is a quasi left (right) duo ring (see \cite[Proposition 2.2]{Yu}). A positive answer to the following question would also imply a positive answer to the preceding one.
	\begin{question}
		Is a left quasi-duo ring with identity a left Dischinger?	
	\end{question}

	\vspace{3mm}
	In Corollary \ref{D.acc}, we proved that $R$ satisfies the ascending chain condition on left annihilators of subsets of $R$. In view of Corollary \ref{D.goldie}, $_RR$ does not have a finite uniform dimension. This can also be seen directly in view of the following result.
	\begin{lemma}\label{D.uniform}
		Suppose $s,\,t \in S$ are two nonzero elements such that $Ss \cap St = 0$ and $l_S(t) = 0$, then the uniform dimension of $_SS$ is infinite.
	\end{lemma}
	\begin{proof}
		We will show that $Ss,\,Sst,\,Sst^2, \ldots$ is an independent family of nonzero left ideals of S. It is clear that each of these ideals is nonzero. Also $s_1s+s_2st+\ldots+s_nst^{n-1} = 0$ implies that $s_1s \in Ss \cap St = 0$. So $s_2st+\ldots+s_nst^{n-1} = 0$ and as $l_S(t) = 0$ we have $s_2s+\ldots+s_nst^{n-2} = 0$. Thus $s_2s \in Ss \cap St = 0$. Proceeding similarly, we get $s_kst^{k-1} = 0$ for all positive integers $k$. 
	\end{proof}
	\begin{cor}
		$_RR$ has infinite uniform dimension.
	\end{cor}
	\begin{proof}
		Both $1-ax$ and $a$ are nonzero elements of $R$, $R(1-ax) \cap Ra = 0$ and $l_R(a) = 0$. So the result follows from Lemma \ref{D.uniform}. 
	\end{proof}
	The following question is natural to ask.
	\begin{question}
		Let $S$ be any ring and $_SS$ have finite uniform dimension. Is $S$ left or right Dischinger?
	\end{question}
\section{Acknowledgements}
We are thankful to Pace P. Nielsen for many interesting questions that he raised in private communication. We are also grateful to him for reading the paper and for many useful comments that improved the quality of the paper.

	\bigskip
	\noindent Department of Mathematics \\
	\noindent  Panjab University \\
	\noindent  Chandigarh 160\,014, India\\[2mm]
	{\tt dimple4goyal@gmail.com}\\[1mm]
	{\tt dkhurana@pu.ac.in}\\[1mm]

\end{document}